\documentclass[11pt,oneside]{article}
\usepackage{amsfonts}
\usepackage{amsmath}
\usepackage{amsthm}
\usepackage[mathscr]{euscript}
\usepackage{mathrsfs}
\usepackage{indentfirst}
\usepackage{xcolor}

\allowdisplaybreaks[4] %%让多行公式强制换页%%
\newtheorem{theo}{Theorem}[section]

\newtheorem{lem}{Lemma}[section]

\newtheorem{defi}{Definition}[section]
\newcommand{\be}{\begin{equation}}
\newcommand{\ee}{\end{equation}}
\newcommand\bes{\begin{eqnarray}} \newcommand\ees{\end{eqnarray}}
\newcommand{\bess}{\begin{eqnarray*}}
\newcommand{\eess}{\end{eqnarray*}}

\newcommand\ep{\varepsilon}
\newcommand\kk{\left}
\newcommand\rr{\right}

\newcommand\dd{\displaystyle}
\newcommand\vp{\varphi}
\newcommand\df{\dd\frac}

\newcommand\ty{T_{max}}

\newcommand\nm{\nonumber}
\newcommand\yy{\infty}
\newcommand\nn{\nabla}

\newcommand\pl{\partial}

\newcommand\ii{\int_\Omega}

\newcommand\oo{\Omega}

\newcommand\R{\mathbb{R}}

\newcommand\tr{\Delta}

\newcommand\dv{\frac{d}{dt}}

\newcommand\faa{{\rm for\ all}}

\newcommand\cd{\cdot}

\newcommand\hm{\mathcal{H}}
\begin{document}
\setlength{\baselineskip}{16pt} \pagestyle{myheadings}

\begin{center}{\LARGE\bf Global solutions of a doubly tactic resource }\\[2mm]
{\LARGE\bf consumption model with logistic source}\\[4mm]
 {\Large  Jianping Wang\footnote{Corresponding author. {\sl E-mail}: jianping0215@163.com}}\\[1mm]
{School of Applied Mathematics, Xiamen University of Technology, Xiamen, 361024, China}
\end{center}

\begin{quote}
\noindent{\bf Abstract.} We study a doubly tactic resource consumption model
 \bess
 \left\{\begin{array}{lll}
 u_t=\tr u-\nabla\cd(u\nabla w),\\[1mm]
 v_t=\tr v-\nabla\cd(v\nabla u)+v(1-v^{\beta-1}),\\[1mm]
 w_t=\tr w-(u+v)w-w+r
 \end{array}\right.
 \eess
in a smooth bounded domain $\oo\in\R^2$ with homogeneous Neumann boundary conditions, where $r\in C^1(\bar\Omega\times[0,\infty))\cap L^\infty(\Omega\times(0,\infty))$ is a given nonnegative function fulfilling
\bess
\int_t^{t+1}\ii|\nn\sqrt{r}|^2<\yy\ \ \ \ \ \ for\ all\ t>0.
\eess
It is shown that, firstly, if $\beta>2$, then the corresponding Neumann initial-boundary problem admits a global bounded classical solution. Secondly, when $\beta=2$, the Neumann initial-boundary problem admits a global generalized solution.


 \end{quote}

 \section{Introduction}
 \setcounter{equation}{0} {\setlength\arraycolsep{2pt}

This article focus on a doubly tactic nutrient consumption model
\bes
 \left\{\begin{array}{lll}
  u_t=\tr u-\nabla\cd(u\nabla w),\\
 v_t=\tr v-\nabla\cd(v\nabla u),\\
 w_t=\tr w-(u+v)w-w+r,
 \end{array}\right.\label{1.0}
 \ees
which is proposed in \cite{Tania2012} accounting for social interaction between different species. Here, $u=u(x,t)$ and $v=v(x,t)$ denote the population densities of foragers and scroungers, respectively, and $w=w(x,t)$ represents the nutrient concentration with external resupply $r$. The problem \eqref{1.0} includes two taxis mechanisms, the taxis for $u$ says that the movement of foragers is directed by the higher concentrations of nutrient, while the taxis for $v$ says that intelligent scroungers orient their movement towards higher concentrations of the forager to find the nutrient indirectly. If we use $`\rightarrow`$ to denote the moving favor of the species, the taxis mechanisms in \eqref{1.0} are
\bess
v\rightarrow u\rightarrow w.
\eess
Hence, the problem \eqref{1.0} is also referred to as a cascaded taxis system. Up to now, the mathematical findings on \eqref{1.0} are very limited. An important feature of \eqref{1.0} is that the cascaded taxis may lead to pattern formation that is lacked in the single-taxis model (\cite{Tania2012}). In \cite{Winkler-M3AS2019}, the global existence and stabilization of generalized solutions are shown under an explicit condition linking $r$ and the initial nutrient concentration. Whereas the classical solution exists globally and is uniformly bounded (\cite{TaoW-2019-forager}) for the one-dimensional version of \eqref{1.0}, smallness condition on the initial data or the taxis coefficients is required in the high dimensions (\cite{WW-M3AS2020}).

In nature, since individuals may death or reproduce, it seems more reasonable to involve the degradation and proliferation. The classical choice is the logistic source, which on the other hand can prevent the blow up phenomenon caused by the taxis scheme (cf. \cite{Osaki-NA2002,mw-cpde2010logistic,xiang-jde2015}). This motivates us to consider an interesting problem: how weak a degradation is required to suppress the minimal chemotactic aggregation? Or, does the classical logistic source suffice to deal with the taxis mechanisms? By involving the generalized logistic sources and homogeneous Neumann boundary conditions as well as initial values, we have from \eqref{1.0} that
 \bes
 \left\{\begin{array}{lll}
  u_t=\tr u-\nabla\cd(u\nabla w)+au(1-u^{\alpha-1}),&x\in\Omega,\ \ t>0,\\
 v_t=\tr v-\nabla\cd(v\nabla u)+bv(1-v^{\beta-1}),&x\in\Omega,\ \ t>0,\\
 w_t=\tr w-(u+v)w-w+r,&x\in\Omega,\ \ t>0,\\
\frac{ \partial u}{\pl\nu}=\frac{ \partial v}{\pl\nu}=\frac{ \partial w}{\pl\nu}=0,\ \ &x\in\partial\Omega,\ t>0,\\
  u(x,0)=u_0(x),\ v(x,0)=v_0(x),\ w(x,0)=w_0(x),\ &x\in\Omega,
 \end{array}\right.\label{1.1a}
 \ees
where $\oo\in \R^2$ is a bounded domain with smooth boundary $\pl\oo$ and $a,b>0$ with $\alpha,\beta>1$. It is shown in \cite{Black-M3AS2020} that the problem \eqref{1.1a} admits global generalized solutions under the condition that $\alpha>\sqrt{2}+1$, $\beta>1$ and $\min\{\alpha,\beta\}>\frac{\alpha+1}{\alpha-1}$. If $\alpha,\beta\ge3$ or $2\le \alpha<3$ with $\beta\ge3\alpha/(2\alpha-3)$, then global bounded classical solutions exist (\cite{WW-M3AS2020}). The later condition has been reduced to $2\le \alpha<3,\beta\ge3$ in \cite{mu-2020dcds}. Recently, it is found that $\alpha\ge2,\beta>2$ is sufficient to ensure the global boundedness of the solution and hence improve the conclusions in \cite{WW-M3AS2020,mu-2020dcds} (\cite{wjp-arxiv2021}). We notice that, all of these conclusions require $\alpha\ge2$, namely, the species $u$ has  (super-)logistic source. Studies on the variant of \eqref{1.0} or \eqref{1.1a} can be found in \cite{cao-M3AS2020,cao-tao2021NARWA,Liw-zamp2021,Liu-nwrwa2019,liuz-zamp2020,W-jde2021}.

{\bf Main results.} By letting $v\equiv0$, $a=0$ and $r\equiv0$, the model \eqref{1.1a} becomes a known nutrient-taxis (prey-taxis) system
 \bes
 \left\{\begin{array}{lll}
  u_t=\tr u-\nabla\cd(u\nabla w),&x\in\Omega,\ \ t>0,\\
 w_t=\tr w-uw-w,&x\in\Omega,\ \ t>0,\\
\frac{ \partial u}{\pl\nu}=\frac{ \partial w}{\pl\nu}=0,\ \ &x\in\partial\Omega,\ t>0,\\
  u(x,0)=u_0(x),\ w(x,0)=w_0(x),\ &x\in\Omega,
 \end{array}\right.\label{1.1b}
 \ees
which admits global classical solutions in two dimension (\cite{JinW,Taow-jde-2012,TaoW-2019-jde,mw-2017jde}). It is observed that the logistic source of $u$ is unnecessary for the global solvability of the nutrient-taxis model \eqref{1.1b}. Inspired by this, we may conjecture that the logistic source of $u$ is also unnecessary for ensuring the global solvability of \eqref{1.1a}. Thus, by letting $a=0$ and $b=1$, the problem we shall consider is
 \bes
 \left\{\begin{array}{lll}
  u_t=\tr u-\nabla\cd(u\nabla w),&x\in\Omega,\ \ t>0,\\
 v_t=\tr v-\nabla\cd(v\nabla u)+v(1-v^{\beta-1}),&x\in\Omega,\ \ t>0,\\
 w_t=\tr w-(u+v)w-w+r,&x\in\Omega,\ \ t>0,\\
\frac{ \partial u}{\pl\nu}=\frac{ \partial v}{\pl\nu}=\frac{ \partial w}{\pl\nu}=0,\ \ &x\in\partial\Omega,\ t>0,\\
  u(x,0)=u_0(x),\ v(x,0)=v_0(x),\ w(x,0)=w_0(x),\ &x\in\Omega.
 \end{array}\right.\label{1.1}
 \ees
It is well known that the effective method to deal with the chemotactic cross-diffusion of $u$ is to find a certain quasi-energy feature of the Lyapunov functional defined by
 \bess
F(u,v)=\ii u\ln u+\frac12\ii \frac{|\nn w|^2}{w}.
 \eess
However, in the present situation, the emergence of $v$ destroys the favorable gradient structure induced by $F(u,v)$. We should enhance the degradation rate of $v$ to get some a priori estimates that can be used in constructing a certain quasi-energy feature of $F$ (Lemma \ref{l3.3}). Moreover, since the previous results on the fully parabolic nutrient-taxis model do not involve $r$ (\cite{JinW,mw-2017jde,mw-2017jde}), some restrictions on $r$ may be required in the present setting. As for the initial data $(u_0,v_0,w_0)$, for simplicity we shall assume that
 \bes
 \left\{\begin{array}{lll}
 u_0\in W^{2,\infty}(\Omega)\ {\rm is\ nonnegative\ with}\ u_0\not\equiv0\ {\rm on}\ \bar\Omega\ {\rm and}\ \frac{\pl u_0}{\pl \nu}=0\ {\rm on}\ \partial\Omega,\\
v_0\in W^{2,\infty}(\Omega)\ {\rm is\ nonnegative\ with}\ v_0\not\equiv0\ {\rm on}\ \bar\Omega\ {\rm and}\ \frac{\pl v_0}{\pl \nu}=0\ {\rm on}\ \partial\Omega,\\
w_0\in W^{2,\infty}(\Omega)\ {\rm satisfies}\ w_0>0\ {\rm in}\ \Omega\ {\rm and}\ \frac{\pl w_0}{\pl \nu}=0\ {\rm on}\ \partial\Omega,\\
 \end{array}\right.\label{1.2}
 \ees

Our first result states that if $v$ has super-logistic source (i.e., $\beta>2$), globally defined smooth solutions can always be found, and hence exclude the emergence of possibly singular behavior.

\begin{theo}\label{t1.1}
Let $\beta>2$ and $\oo\in\R^2$ be a bounded domain with smooth boundary. Suppose that $(u_0,v_0,w_0)$ satisfies \eqref{1.2} and the nonnegative function $r$ satisfies
\bes
r\in C^1(\bar\Omega\times[0,\infty))\cap L^\infty(\Omega\times(0,\infty))\label{1.7a}
\ees
and
\bes
\int_t^{t+1}\ii|\nn\sqrt{r}|^2<\yy\ \ \ \ \ \ for\ all\ t>0.\label{1.7}
\ees
Then the problem \eqref{1.1} admits a unique, global classical solution $(u,v,w)$ fulfilling
\bess
u,v,w\in C^{2,1}(\bar\oo\times(0,\yy)).
\eess
Moreover, this solution is nonnegative and bounded in $\bar\oo\times(0,\yy)$.
\end{theo}

In the case of $\beta=2$ (namely, the classical logistic source), it is shown that the problem \eqref{1.1} is globally solvable in a natural generalized framework.
\begin{theo}\label{t1.2}
Let $\beta=2$ and $\oo\in\R^2$ be a bounded domain with smooth boundary. Assume that $(u_0,v_0,w_0)$ satisfies \eqref{1.2} and $r$ is a nonnegative function fulfilling \eqref{1.7a} and \eqref{1.7}. Then the problem \eqref{1.1} possesses a global weak solution in the sense of Definition \ref{d1.1}.
\end{theo}
Under the condition that $r$ satisfies \eqref{1.7}, Theorem \ref{t1.1} improves the previous results in \cite{WW-M3AS2020,wjp-arxiv2021,mu-2020dcds} and Theorem \ref{t1.2} partially improves the result in \cite{Black-M3AS2020}. We remark that \eqref{1.7} is a technical restriction that may emerge even in the nutrient-taxis model \eqref{1.1b}. If $r$ is a nonnegative constant, then \eqref{1.7a} and \eqref{1.7} are fulfilled. Of course, if we drop $r$ and suppose that the intra-specific interaction of $w$ is of logistic type, namely,
\bess
w_t=\tr w-(u+v)w+\mu(w-kw),
\eess
then the conclusion holds as well.

\section{Existence of local solutions and some preliminaries}
\setcounter{equation}{0} {\setlength\arraycolsep{2pt}
We shall find some a priori estimates for a general system that can be used in the construction of both the classical solutions and generalized solutions. Let us consider the following system more general than \eqref{1.1}, that is
 \bes
 \left\{\begin{array}{lll}
  u_t=\tr u-\nabla\cd(uF'(u)\nabla w),&x\in\Omega,\ \ t>0,\\
 v_t=\tr v-\nabla\cd(v\nabla u)+v(1-v^{\beta-1}),&x\in\Omega,\ \ t>0,\\
 w_t=\tr w-F(u)w-F(v)w-w+r,&x\in\Omega,\ \ t>0,\\
\frac{ \partial u}{\pl\nu}=\frac{ \partial v}{\pl\nu}=\frac{ \partial w}{\pl\nu}=0,\ \ &x\in\partial\Omega,\ t>0,\\
  u(x,0)=u_0(x),\ v(x,0)=v_0(x),\ w(x,0)=w_0(x),\ &x\in\Omega,
 \end{array}\right.\label{2.1b}
 \ees
where
\bess
F(s)=\left\{\begin{array}{lll}
 s,&\beta>2,\\
 F_\ep(s)=\frac{s}{1+\ep s}\ \ {\rm for}\ \ep>0,\ &\beta=2.
 \end{array}\right.
\eess
Clearly, $0\le F'(s)\le 1$ and $0\le F(s)\le s$ for all $s\ge0$. We will show that, in the case $\beta\ge2$, problem \eqref{2.1b} admits a unique global classical solution. Especially, problem \eqref{2.1b} becomes our original model \eqref{1.1} in the case of $\beta>2$.

The following statement on local existence of solutions to \eqref{2.1b} is proved in \cite[Lemma 2.1]{TaoW-2019-forager}.
\begin{lem}\label{l2.1}\, Let $\beta>1$ and $\oo\in\R^2$ be a bounded domain with smooth boundary. Suppose that $r$ is a nonnegative function fulfilling \eqref{1.7a} and $(u_0,v_0,w_0)$ satisfies \eqref{1.2}. Then there exist $\ty\in(0,\yy]$ and nonnegative functions
\bess
 u,v,w\in \bigcap_{q>2}C^0([0,\ty);W^{1,q}(\oo))\cap C^{2,1}(\bar\oo\times(0,\ty))
\eess
such that $(u,v,w)$ solves \eqref{2.1b} uniquely in the classical sense and satisfies $u>0,v>0$ and $w>0$ in $\oo\times(0,\ty)$. Moreover, if $\ty<\infty$, then
\[ \limsup_{t\to \ty}\kk(\|u(\cdot,t)\|_{W^{1,p}(\oo)}+\|v(\cdot,t)\|_{W^{1,p}(\oo)}+\|w(\cdot,t)\|_{W^{1,p}(\oo)}\rr)=\yy\ \ { for\ all}\ p>2.
 \]
\end{lem}

The existence of $\ty$ enables us to set $\tau=\min\{\frac{\ty}{2},1\}$. It is noted that $\ty$ may depend on $\ep$ in the case of $\beta=2$. The following $L^\yy(\oo)$ bound for $w$ and $L^1(\oo)$-boundedness of $u$ and $v$ as well as the space-time integral estimate of $v$ have been established in \cite{TaoW-2019-forager,wjp-arxiv2021}.
\begin{lem}\label{l2.2}
Whenever $\beta>1$, for any choice of $r\ge0$ fulfilling \eqref{1.7a}, the solution $(u,v,w)$ of \eqref{2.1b} satisfies
\bes
\|w(\cdot,t)\|_\yy\le M\ \ \ \ \ for\ all\ t\in(0,\ty),\label{2.2}
\ees
for $M=\|w_0\|_\yy+r_*$ with $r_*:=\|r\|_{L^\yy(\oo\times(0,\yy))}$, and
\bes
\ii u&=& M_1:=\ii u_0\ \ \ \ \ for\ all\ t\in(0,\ty). \label{2.3}
 \ees
Moreover, there exist $C>0$ independent of $\ep$ such that
\bes
\ii v\le C\ \ \ \ \ for\ all\ t\in(0,\ty), \label{2.4}
\ees
and
\bes
\int_t^{t+\tau}\ii v^{\beta}\le C\ \ \ \ \ for\ all\ t\in(0,\ty-\tau). \label{2.4a}
\ees
\end{lem}
\begin{proof}\, The inequality \eqref{2.2} has been proved in \cite[Lemma 2.2]{TaoW-2019-forager}. The assertions \eqref{2.3}-\eqref{2.4a} can be obtained by integrating the first equation and second equation in \eqref{1.1} over $\oo$, respectively (cf. \cite{wjp-arxiv2021}).
\end{proof}

Let us introduce a fundamental ODE inequality that will be used later (see \cite[Lemma 2.4]{SunL-jmaa2020} for a similar version).
\begin{lem}\label{l2.4}
Let $a,b,c>0$. Assume that for some $T\in(0,\infty]$ and $\theta=\min\{1,T/2\}$, the nonnegative functions $y\in C([0,T))\cap C^1((0,T))$, $z\in L_{loc}^1([0, T))$ and satisfy
 \bes
 y'(t)\le ay(t)z(t)\ \ \ \ \ \ for\ all\ t\in(0, T),\label{2.10}
\ees
and
\bes
 &\int_t^{t+\theta}y(s){\rm d}s\le b\ \ \ \ \ \ for\ all\ t\in(0, T-\theta)&\label{2.12}\\
&\int_t^{t+\theta}z(s){\rm d}s\le c\ \ \ \ \ \ for\ all\ t\in(0, T-\theta)&,\label{2.11}
\ees
Then
\bes
y(t)\le \max\{y(0),b\}e^{2ac}\ \ \ \ \ \ for\ all\ t\in(0,T).\label{2.13}
\ees
\end{lem}
\begin{proof} By \eqref{2.10}, it is easy to get that, for $0\le t_0<t<T$,
\bes
y(t)\le y(t_0)e^{a\int_{t_0}^tz(s){\rm d}s}\label{2.14}
\ees
We consider two cases separately according to the value of $T$.

{\it Case I}: $T< 2$, i.e., $T=2\theta$. It follows from $0<\theta\le 1$, \eqref{2.14} and \eqref{2.11} that,
\bes
y(t)\le y(0)e^{a\int_{0}^t z(s){\rm d}s}\le y(0)e^{a\int_{0}^\theta z(s){\rm d}s}\le y(0)e^{ac}
\ \ \ \ \ \ \faa\ t\in(0,\theta],\label{2.15}
\ees
and hence
\bes
y(t)\le y(\theta)e^{a\int_{\theta}^t z(s){\rm d}s}\le y(0)e^{ac}e^{a\int_{\theta}^{2\theta} z(s){\rm d}s}\le y(0)e^{2ac}
\ \ \ \ \ \ \faa\ t\in(\theta,T).\label{2.16}
\ees
Therefore, we have \eqref{2.13} from \eqref{2.15} and \eqref{2.16} in the case of $T<2$.

{\it Case II}: $T\ge 2$, i.e., $\theta=1$. For any $t\in(1,T)$, by the mean value theorem, we infer from \eqref{2.12} that, there is $t_*\in[t-1,t]$ such that
\bess
y(t_*)=\int_t^{t+1}y(s){\rm d}s\le b.
\eess
This combined with \eqref{2.14} gives
\bes
y(t)\le y(t_*)e^{a\int_{t_*}^tz(s){\rm d}s}\le y(t_*)e^{a\int_{t-1}^tz(s){\rm d}s}\le be^{ac}\ \ \ \ \ \ \faa\ t\in(1,T).\label{2.17}
\ees
By \eqref{2.15}, we moreover have
\bes
y(t)\le y(0)e^{a\int_{0}^t z(s){\rm d}s}\le y(0)e^{a\int_{0}^1 z(s){\rm d}s}\le y(0)e^{ac}
\ \ \ \ \ \ \faa\ t\in(0,1].\label{2.18}
\ees
Thus, we obtain \eqref{2.13} from \eqref{2.17} and \eqref{2.18}.
\end{proof}

\section{Energy estimates}
\setcounter{equation}{0} {\setlength\arraycolsep{2pt}
The system \eqref{1.1} enjoys a favorable quasi-energy functional structure
\bess
\dv\kk\{\ii u\ln u+\frac12\ii \frac{|\nn w|^2}{w}\rr\}+\kk\{\ii u\ln u+\frac12\ii \frac{|\nn w|^2}{w}\rr\}+\frac12\ii \frac{|\nn u|^2}{u}
\le-\ii \nn F(v)\cdot\nn w+C\nm
\eess
for some $C>0$, in which we have eliminated the taxis driven term $\nn F(u)\cdot\nn w$. However, there is a new essential difficulty in dealing with $\ii \nn F(v)\cdot\nn w$. By an integration by parts and Young's inequality, $\ii \nn F(v)\cdot\nn w=\ii F(v)\tr w$ can be controlled by $\ii v^2$ and $\ii |\tr w|^2$. Due to $\beta\ge2$, we may involve \eqref{2.4a} to handle $\ii v^2$. To absorb $\ii |\tr w|^2$, we shall introduce the estimation of $\ii|\nn w|^2$, i.e.,
 \bess
\df12\dv\ii|\nn w|^2+\ii|\nn w|^2+\frac14\ii|\tr w|^2\le M^2\ii u^2+M^2\ii v^2+r^2|\oo|,\nonumber
\eess
where we use \eqref{2.4a} to control $\ii v^2$, and $\ii u^2$ can be estimated by $\ii \frac{|\nn u|^2}{u}$ by the Gagliardo-Nirenberg inequality. So, we can get the space-time integral regularities of $u$ and $\tr w$.

\begin{lem}\label{l3.1}
Let $\beta>1$, and suppose that $r$ is a nonnegative function satisfying \eqref{1.7a}. Then there exists $C>0$ independent of $\ep$ such that the solution of \eqref{2.1b} fulfills
\bes
&&\df12\dv\ii|\nn w|^2+\ii|\nn w|^2+\frac14\ii|\tr w|^2\nonumber\\
&\le&C\ii\frac{|\nn u|^2}{u}+M^2\ii v^2+C\ \ \ \ \ \ \faa\ t\in(0,\ty).\label{3.2h}
\ees
\end{lem}
\begin{proof}
The Gagliardo-Nirenberg inequality yields $C_1>0$ fulfilling
\bess
\|\vp\|_{L^4(\oo)}^4\le C_1\|\nn \vp\|_{L^2(\oo)}^2\|\vp\|_{L^2(\oo)}^2+C_1\|\vp\|_{L^2(\oo)}^4\ \ \ \ \ \ \faa\ \vp\in W^{1,2}(\oo).
\eess
Due to \eqref{2.3}, this implies
\bes
\ii u^2=\|u^{\frac12}\|_{L^4(\oo)}^4&\le& C_1\|\nn u^{\frac12}\|_{L^2(\oo)}^2\|u^{\frac12}\|_{L^2(\oo)}^2+C_1\|u^{\frac12}\|_{L^2(\oo)}^4\nm\\
&\le& C_1M_1\ii\frac{|\nn u|^2}{u}+C_1M_1^2\ \ \ \ \ \ \faa\ t\in(0,\ty),\label{3.2g}
\ees
By using the integration by parts, Young's inequality and \eqref{3.2g}, we obtain
\bess
&&\df12\dv\ii|\nn w|^2+\ii|\nn w|^2\nonumber\\
&=&\ii\nn w\cdot\nn w_t+\ii|\nn w|^2\nonumber\\
&=&\ii\nn w\cdot\nn(\Delta w-F(u)w-F(v)w- w+r)+\ii|\nn w|^2\nonumber\\
&=&\ii\nn w\cdot\nn\Delta w-\ii\nn w\cdot\nn(F(u)w)-\ii\nn w\cdot\nn(F(v)w)+\ii\nn w\cdot\nn r\nonumber\\
&=&-\ii|\tr w|^2+\ii F(u)w\tr w+\ii F(v)w\tr w-\ii r\tr w\nonumber\\
&\le&-\frac14\ii|\tr w|^2+M^2\ii u^2+M^2\ii v^2+r_*^2|\oo|\nonumber\\
&\le&-\frac14\ii|\tr w|^2+M^2\kk(C_1M_1\ii\frac{|\nn u|^2}{u}+C_1M_1^2\rr)+M^2\ii v^2+r_*^2|\oo|\nonumber\\
&\le&-\frac14\ii|\tr w|^2+C_2\ii\frac{|\nn u|^2}{u}+M^2\ii v^2+C_2\ \ \ \ \ \ \faa\ t\in(0,\ty),
\eess
where $C_2=\max\{C_1M^2M_1,C_1M^2M_1^2+r_*^2|\oo|\}$ and $r_*:=\|r\|_{L^\yy(\oo\times(0,\yy))}$, which gives \eqref{3.2h}.
\end{proof}

\begin{lem}\label{l3.2}
Let $\beta>1$ and $r$ be a nonnegative function fulfilling \eqref{1.7a}. Then one can find $C>0$ independent of $\ep$ such that for any $\delta>0$
\bes
&&\dv\kk\{\ii \kk(u\ln u+\frac{1}{e}\rr)+\frac12\ii \frac{|\nn w|^2}{w}\rr\}+\kk\{\ii \kk(u\ln u+\frac{1}{e}\rr)+\frac12\ii \frac{|\nn w|^2}{w}\rr\}+\frac12\ii \frac{|\nn u|^2}{u}\nm\\
&\le& \delta\ii|\tr w|^2+\frac{1}{4\delta}\ii v^2+2\ii|\nn\sqrt{r}|^2+C\ \ \ \ \ \ \faa\ t\in(0,\ty),\label{3.10}
\ees
\end{lem}
\begin{proof}
Following the arguments in \cite[Lemma 3.2 and Lemma 3.3]{JinW} (or \cite{mw-2012cpde,mw-2017jde}), we infer that
\bes
\dv\ii u\ln u+\ii u\ln u\nm&=&-\ii \frac{|\nn u|^2}{u}+\ii u\ln u+\ii \nn F(u)\cdot\nn w\nm\\
&\le& -\frac12\ii \frac{|\nn u|^2}{u}+\ii \nn F(u)\cdot\nn w+C_1\ \ \ \ \ \ \faa\ t\in(0,\ty).\label{3.3a}
\ees
for some $C_1>0$ and
\bes
&&\frac12\dv\ii \frac{|\nn w|^2}{w}+\frac12\ii\frac{|\nn w|^2}{w}\nm\\
&=&-\ii w|D^2\ln w|^2+\frac12\int_{\pl\oo}\frac1{w}\cdot\frac{\partial|\nn w|^2}{\partial\nu}-\ii \nn F(u)\cdot\nn w-\ii \nn F(v)\cdot\nn w\nm\\
&&-\frac12\ii F(u)\frac{|\nn w|^2}{w}-\frac12\ii F(v)\frac{|\nn w|^2}{w}+\ii \nn r\cdot\frac{\nn w}{w}-\frac{1}2\ii r\frac{|\nn w|^2}{w^2}\nm\\
&\le&-\ii w|D^2\ln w|^2+\frac12\int_{\pl\oo}\frac1{w}\cdot\frac{\partial|\nn w|^2}{\partial\nu}-\ii \nn F(u)\cdot\nn w-\ii \nn F(v)\cdot\nn w\nm\\
&&+\ii \nn r\cdot\frac{\nn w}{w}-\frac{1}2\ii r\frac{|\nn w|^2}{w^2}\nm\\
&\le& -\ii \nn F(u)\cdot\nn w-\ii \nn F(v)\cdot\nn w+\ii \nn r\cdot\frac{\nn w}{w}-\frac{1}2\ii r\frac{|\nn w|^2}{w^2}+C_2\label{3.2a}
\ees
for all $t\in(0,\ty)$ with $C_2>0$. Employing Young's inequality yields
\bes
&&\ii \nn r\cdot\frac{\nn w}{w}-\frac{1}2\ii r\frac{|\nn w|^2}{w^2}\nm\\
&=&2\ii\sqrt{r}\nn\sqrt{r}\cdot\frac{\nn w}{w}-\frac{1}2\ii r\frac{|\nn w|^2}{w^2}\nm\\
&\le&2\ii|\nn\sqrt{r}|^2.\nm
\ees
Inserting this into \eqref{3.2a} provides
\bes
&&\frac12\dv\ii \frac{|\nn w|^2}{w}+\frac12\ii\frac{|\nn w|^2}{w}\nm\\
&\le& -\ii \nn F(u)\cdot\nn w-\ii \nn F(v)\cdot\nn w+2\ii|\nn\sqrt{r}|^2+C_2\label{3.3b}
\ees
By combining \eqref{3.3b} and \eqref{3.3a} and using Young's inequality, there exists $C_3>0$ such that
\bes
&&\dv\kk\{\ii u\ln u+\frac12\ii \frac{|\nn w|^2}{w}\rr\}+\kk\{\ii u\ln u+\frac12\ii \frac{|\nn w|^2}{w}\rr\}+\frac12\ii \frac{|\nn u|^2}{u}\nm\\
&\le&-\ii \nn F(v)\cdot\nn w+2\ii|\nn\sqrt{r}|^2+C_3\nm\\
&=& \ii F(v)\tr w+2\ii|\nn\sqrt{r}|^2+C_3\nm\\
&\le& \delta\ii|\tr w|^2+\frac{1}{4\delta}\ii v^2+2\ii|\nn\sqrt{r}|^2+C_3\ \ \ \ \ \ \faa\ t\in(0,\ty)\nm
\ees
holds for any $\delta>0$, or equivalently
\bess
&&\dv\kk\{\ii \kk(u\ln u+\frac{1}{e}\rr)+\frac12\ii \frac{|\nn w|^2}{w}\rr\}+\kk\{\ii \kk(u\ln u+\frac{1}{e}\rr)+\frac12\ii \frac{|\nn w|^2}{w}\rr\}+\frac12\ii \frac{|\nn u|^2}{u}\nm\\
&\le& \delta\ii|\tr w|^2+\frac{1}{4\delta}\ii v^2+2\ii|\nn\sqrt{r}|^2+C_4\ \ \ \ \ \ \faa\ t\in(0,\ty),
\eess
where $C_4=C_3+\frac{|\oo|}{e}$. This completes the proof.
\end{proof}

\begin{lem}\label{l3.3}
Suppose that $\beta\ge2$ and $r$ is a nonnegative function fulfilling \eqref{1.7a} and \eqref{1.7}. Then there exists $C>0$ independent of $\ep$ such that the solution of \eqref{2.1b} satisfies
\bes
\ii |\nn w|^2\le C\ \ \ \ \ \ for\ all\ t\in(0,\ty)\label{3.1b}
\ees
and
\bes
\int_t^{t+\tau}\ii u^2\le C\ \ \ \ \ \ for\ all\ t\in(0,\ty-\tau)\label{3.1e}
\ees
as well as
\bes
\int_t^{t+\tau}\ii|\tr w|^2\le C\ \ \ \ \ \ for\ all\ t\in(0,\ty-\tau).\label{3.1f}
\ees
\end{lem}
\begin{proof}
From \eqref{3.2h}, there is $C_1>0$ fulfilling
\bes
&&\df12\dv\ii|\nn w|^2+\ii|\nn w|^2+\frac14\ii|\tr w|^2\nonumber\\
&\le&C_1\ii\frac{|\nn u|^2}{u}+M^2\ii v^2+C_1\ \ \ \ \ \ \faa\ t\in(0,\ty).\label{3.5}
\ees
By taking $\delta=\frac{1}{32C_1}$, we have from \eqref{3.10} that
\bes
&&\dv\kk\{\ii \kk(u\ln u+\frac{1}{e}\rr)+\frac12\ii \frac{|\nn w|^2}{w}\rr\}+\kk\{\ii \kk(u\ln u+\frac{1}{e}\rr)+\frac12\ii \frac{|\nn w|^2}{w}\rr\}+\frac12\ii \frac{|\nn u|^2}{u}\nm\\
&\le& \frac{1}{32C_1}\ii|\tr w|^2+8C_1\ii v^2+2\ii|\nn\sqrt{r}|^2+C_2\ \ \ \ \ \ \faa\ t\in(0,\ty)\label{3.10a}
\ees
for some $C_2>0$. Multiplying \eqref{3.5} by $\frac{1}{4C_1}$ and adding the obtained result to \eqref{3.10a} gives
\bes
&&\dv\kk\{\ii \kk(u\ln u+\frac{1}{e}\rr)+\frac12\ii \frac{|\nn w|^2}{w}+\frac{1}{8C_1}\ii|\nn w|^2\rr\}\nm\\
&&+\kk\{\ii \kk(u\ln u+\frac{1}{e}\rr)+\frac12\ii \frac{|\nn w|^2}{w}+\frac{1}{8C_1}\ii|\nn w|^2\rr\}+\frac14\ii \frac{|\nn u|^2}{u}+\frac{1}{32C_1}\ii|\tr w|^2\nm\\
&\le& \kk(8C_1+\frac{M^2}{4C_1}\rr)\ii v^2+2\ii|\nn\sqrt{r}|^2+C_2+\frac14\nm\\
&\le&C_3\ii v^2+2\ii|\nn\sqrt{r}|^2+C_3\ \ \ \ \ \ \faa\ t\in(0,\ty),\label{3.2i}
\ees
where $C_3=\max\{8C_1+\frac{M^2}{4C_1},C_2+\frac14\}$. By letting
$$y(t)=\ii \kk(u\ln u+\frac{1}{e}\rr)+\frac12\ii \frac{|\nn w|^2}{w}+\frac{1}{8C_1}\ii|\nn w|^2$$
 and
 $$z(t):=C_3\ii v^2+2\ii|\nn\sqrt{r}|^2+C_3,$$
we have from \eqref{3.2i} that firstly
\bes
y'(t)+y(t)+\frac14\ii \frac{|\nn u|^2}{u}+\frac{1}{32C_1}\ii|\tr w|^2\le z(t)\ \ \ \ \ \ \faa\ t\in(0,\ty)\label{3.11a}
\ees
and that in the second place
\bes
y'(t)+y(t)\le z(t)\ \ \ \ \ \ \faa\ t\in(0,\ty).\label{3.11}
\ees
Clearly, $y(t),z(t)\ge0$ for all $t\in(0,\ty)$, and due to \eqref{2.4a} and \eqref{1.7} there exists $C_4>0$ such that
\bes
\int_t^{t+\tau}z(s){\rm d}s\le C_4\ \ \ \ \ \ \faa\ t\in(0,\ty-\tau).\label{3.2l}
\ees
Hence, according to \eqref{3.11} and \eqref{3.2l}, we can apply \cite[Lemma 2.2]{WW-M3AS2020} to find $C_5>0$ such that
\bes
y(t)\le C_5\ \ \ \ \ \ \faa\ t\in(0,\ty),\label{3.2j}
\ees
which gives \eqref{3.1b}. Making use of \eqref{3.2l}, \eqref{3.2j} and the nonnegativity of $y(t)$ and $z(t)$, we integrate \eqref{3.11a} over $(t,t+\tau)$ to find $C_7>0$ fulfilling
\bess
\int_t^{t+\tau}\ii \frac{|\nn u|^2}{u}+\int_t^{t+\tau}\ii|\tr w|^2\le C_7\ \ \ \ \ \ \faa\ t\in(0,\ty-\tau).
\eess
This provides \eqref{3.1f}. Moreover, recalling \eqref{3.2g}, we obtain \eqref{3.1e} from the above inequality. The proof is end.
\end{proof}

Based on \eqref{3.1e} and \eqref{3.1f}, we get the uniform-in-time $L^p(\oo)$-boundedness of $u$ for any $p>1$ and space-time $L^2$ integral regularity of $\nn u$.
\begin{lem}\label{l3.4}
Suppose that $\beta\ge2$ and $r$ is a nonnegative function fulfilling \eqref{1.7a} and \eqref{1.7}. For any $p>1$, one can find $C=C(p)>0$ independent of $\ep$ such that the solution component $u$ of \eqref{2.1b} satisfies
\bes
\ii u^p\le C\ \ \ \ \ \ for\ all\ t\in(0,\ty).\label{3.1g}
\ees
Moreover, there is $C'>0$ independent of $\ep$ satisfying
\bes
\int_t^{t+\tau}\ii |\nn u|^2\le C'\ \ \ \ \ \ for\ all\ t\in(0,\ty-\tau).\label{3.1i}
\ees
\end{lem}
\begin{proof}
Let $p_i=2^i$ with $i\in\mathbb{N}$ and $\hm_i(s)=\int_0^s \sigma^{p_i-1}F'(\sigma){\rm d}\sigma$. Clearly, we have $p_{i+1}=2p_i$ and $p_0=1$. Since $0\le F'(\sigma)\le 1$ for any $\sigma\ge0$, it is easy to see that
\bes
0\le\hm_i(s)\le \int_0^s \sigma^{p_i-1}{\rm d}\sigma\le \frac{s^{p_i}}{p_i}\ \ \faa\ s\ge0.\label{3.15}
\ees
An application of the known Gagliardo-Nirenberg inequality provides $C_1>0$ fulfilling
\bes
\|u^{p_{i-1}}\|_{L^4(\oo)}^2\le C_1\|\nn u^{p_{i-1}}\|_{L^2(\oo)}\|u^{p_{i-1}}\|_{L^2(\oo)}+C_1\|u^{p_{i-1}}\|_{L^1(\oo)}^2.\label{3.16}
\ees
Testing the first equation of \eqref{2.1b} by $u^{p_i-1}$ provides
\bes
&&\frac{1}{p_i}\dv\ii u^{p_i}+\frac{4(p_i-1)}{p_i^2}\ii|\nabla u^{p_{i-1}}|^2\nm\\
&=&(p_i-1)\ii u^{p_i-1}F'(u)\nn u\cdot\nn w\nm\\
&=&-(p_i-1)\ii\hm_i(u)\tr w\nm\\
&\le& \ii u^{p_i}|\tr w|\nm\\
&\le&\kk(\ii u^{p_{i+1}}\rr)^{\frac12}\kk(\ii|\tr w|^2\rr)^{\frac12}\nm\\
&=&\|u^{p_{i-1}}\|_{L^4(\oo)}^2\|\tr w\|_{L^2(\oo)}\nm\\
&\le&C_1\kk(\|\nn u^{p_{i-1}}\|_{L^2(\oo)}\|u^{p_{i-1}}\|_{L^2(\oo)}+\|u^{p_{i-1}}\|_{L^1(\oo)}^2\rr)\|\tr w\|_{L^2(\oo)}\nm\\
&\le& \frac{3(p_i-1)}{p_i^2}\|\nn u^{p_{i-1}}\|_{L^2(\oo)}^2+C_2\|u^{p_{i-1}}\|_{L^2(\oo)}^2\|\tr w\|_{L^2(\oo)}^2+C_1\|u^{p_{i-1}}\|_{L^1(\oo)}^2\|\tr w\|_{L^2(\oo)}\nm\\
&=&\frac{3(p_i-1)}{p_i^2}\ii|\nabla u^{p_{i-1}}|^2+C_2\ii u^{p_i}\cdot\ii|\tr w|^2+C_1\kk(\ii u^{p_{i-1}}\rr)^2\cdot\kk(\ii|\tr w|^2\rr)^{\frac12},\nm
\ees
i.e.,
\bes
&&\dv\ii u^{p_i}+\frac{1}{p_i}\ii|\nabla u^{p_{i-1}}|^2\nm\\
&\le&\dv\ii u^{p_i}+\frac{p_i-1}{p_i}\ii|\nabla u^{p_{i-1}}|^2\nm\\
&\le& C_2p_i\ii u^{p_i}\cdot\ii|\tr w|^2+C_1p_i\kk(\ii u^{p_{i-1}}\rr)^2\cdot\kk(\ii|\tr w|^2\rr)^{\frac12}\nm\\
&\le&C_2p_i\ii u^{p_i}\cdot\ii|\tr w|^2+C_1p_i\ii|\tr w|^2+C_1p_i\kk(\ii u^{p_{i-1}}\rr)^4\nm\\
&\le&C_3\kk(\ii u^{p_i}+1\rr)\kk(\ii|\tr w|^2+\kk(\ii u^{p_{i-1}}\rr)^4\rr)\label{3.17}
\ees
with $C_3=\max\{C_1p_i,C_2p_i\}$ for all $i\ge1$ and $t\in(0,\ty)$ where we have used \eqref{3.15}, \eqref{3.16}, H\"{o}lder's inequality and Young's inequality. Taking
\bess
y_i(t)=\ii u^{p_i}+1,\ \ z_i=\ii|\nabla u^{p_{i}}|^2\ \ {\rm and}\ \ h(t)=\ii|\tr w|^2.
\eess
Then \eqref{3.17} says firstly
\bes
y_i'(t)+\frac{1}{p_{i}}z_{i-1}(t)\le C_3y_i(t)\kk(h(t)+y_{i-1}^4(t)\rr)\ \ \ \ \ \ \faa\ t\in(0,\ty).\label{3.19a}
\ees
and in the second place
\bes
y_i'(t)\le C_3y_i(t)\kk(h(t)+y_{i-1}^4(t)\rr)\ \ \ \ \ \ \faa\ t\in(0,\ty).\label{3.19}
\ees
Moreover, we know from \eqref{3.1f} that, there is $C_4>0$ such that
\bes
\int_t^{t+\tau} h(s){\rm d}s\le C_4\ \ \ \ \ \ \faa\ t\in(0,\ty-\tau).\label{3.21}
\ees

We claim that, if
\bes
y_{i-1}(t)< K_{i-1}\ \ \ \ \ \ \faa\ t\in(0,\ty)\label{3.21c}
\ees
and
\bes
\int_t^{t+\tau} y_i(s){\rm d}s\le k_i\ \ \ \ \ \ \ \faa\ t\in(0,\ty-\tau)\label{3.21b}
\ees
for some $K_{i-1},k_i>0$, then
\bess
\max_{0<t<\ty}y_i(t)< \yy\ \ {\rm and}\ \ \max_{0<t<\ty-\tau}\kk\{\int_t^{t+\tau}y_{i+1}(s){\rm d}s\rr\}<\yy.
\eess
Actually, by \eqref{3.21}, \eqref{3.21c} and the fact that $0<\tau\le 1$, we find
\bes
\int_{t}^{t+\tau}\kk(h(s)+y_{i-1}^4(s)\rr){\rm d}s\le C_4+K_{i-1}^4\ \ \ \ \ \ \faa\ t\in(0,\ty-\tau).\label{3.22}
\ees
Due to \eqref{3.19}, \eqref{3.21b} and \eqref{3.22}, we can use Lemma \ref{l2.4} to deduce that
\bes
\max_{0<t<\ty}y_i(t)< \yy.\label{3.19c}
\ees
Moreover, integrating \eqref{3.19a} upon $(t,t+\tau)$ and using \eqref{3.22} and \eqref{3.19c}, one has
\bes
\max_{0<t<\ty-\tau}\kk\{\int_t^{t+\tau}z_{i-1}(s){\rm d}s\rr\}<\yy.\label{3.21f}
\ees
Since $p_{i+1}=4p_{i-1}$ and $p_i=2p_{i-1}$, we recall from \eqref{3.16} that
\bess
\ii u^{p_{i+1}}\le 2C_1\ii|\nn u^{p_{i-1}}|^2\cdot\ii u^{p_i}+2C_1\kk(\ii u^{p_{i-1}}\rr)^4,
\eess
i.e.,
\bess
y_{i+1}(t)\le 2C_1z_{i-1}(t)y_i(t)+2C_1y_{i-1}(t)^4.
\eess
This in conjunction with \eqref{3.21c}, \eqref{3.19c} and \eqref{3.21f}  shows
\bess
\max_{0<t<\ty-\tau}\kk\{\int_t^{t+\tau}y_{i+1}(s){\rm d}s\rr\}<\yy.
\eess

When $i=1$, from \eqref{2.3} and \eqref{3.1e} we know that
\bess
y_0(t)=\ii u(\cdot,t)=\ii u_0=M_1\ \ \ \ \ \ \faa\ t\in(0,\ty)
\eess
and for some $C_5>0$,
\bess
\int_t^{t+\tau} y_1(s){\rm d}s=\int_t^{t+\tau}\ii u^2\le C_5\ \ \ \ \ \ \faa\ t\in(0,\ty-\tau),
\eess
which shows that, there is $\mathcal{K}_1>0$ such that
\bes
y_1(t)< \mathcal{K}_1\ \ \ \ \ \ \faa\ t\in(0,\ty),\label{3.21d}
\ees
and
\bes
\int_t^{t+\tau} y_2(s){\rm d}s< \mathcal{K}_1\ \ \ \ \ \ \faa\ t\in(0,\ty-\tau).\label{3.21e}
\ees
Again by the same arguments, we have from \eqref{3.21d} and \eqref{3.21e} that, for some $\mathcal{K}_2>0$
\bess
y_2(t)< \mathcal{K}_2\ \ \ \ \ \ \faa\ t\in(0,\ty),
\eess
and
\bess
\int_t^{t+\tau} y_3(s){\rm d}s< \mathcal{K}_2\ \ \ \ \ \ \faa\ t\in(0,\ty-\tau).
\eess
Proceeding inductively, we can see that for any $i\in\mathbb{N}$, there is $\mathcal{K}_i>0$ such that $y_i(t)\le \mathcal{K}_i$ for all $t\in(0,\ty)$ and $\int_t^{t+\tau} y_{i+1}(s){\rm d}s< \mathcal{K}_i$ for all $t\in(0,\ty-\tau)$. Recalling the definition of $y_i(t)$, this shows \eqref{3.1g}. Moreover, from the proof, we know that $\int_t^{t+\tau} z_0(s){\rm d}s<\yy$, which gives \eqref{3.1i}.
\end{proof}

\section{Global existence of classical solutions: proof of Theorem 1.1}
\setcounter{equation}{0} {\setlength\arraycolsep{2pt}
The regularized problem \eqref{2.1b} is same with our original model \eqref{1.1} in the case of $\beta>2$. We recall a boundedness criterion from \cite[Proposition 3.1 and Remark 3.1]{wjp-arxiv2021} to determine the global solvability of \eqref{1.1} or \eqref{2.1b}.
\begin{lem}\label{l2.5}
Let $\beta>2$, $\oo\in\R^2$ be a bounded domain with smooth boundary and $r\ge0$ satisfy \eqref{1.7a}. Suppose that for some $\bar p,\bar q>2$ and $K_1,K_2>0$, the solution of \eqref{2.1b} fulfills
\bess
\int_t^{t+\tau}\ii u^{\bar q}\le K_1\ \ \ \ \ for\ all\ t\in(0,\ty-\tau)
\eess
and
\bess
\int_t^{t+\tau}\ii v^{\bar p}\le K_2\ \ \ \ \ for\ all\ t\in(0,\ty-\tau).
\eess
Then $\ty=\yy$, and there exist $\theta\in(0,1)$ and $C>0$ fulfilling
\bess
&&\|u\|_{C^{2+\theta,1+\frac{\theta}2}(\bar\oo\times[t,t+1])}+\|v\|_{C^{2+\theta,1+\frac{\theta}2}(\bar\oo\times[t,t+1])}\nm\\
&&\quad\quad+\|w\|_{C^{2+\theta,1+\frac{\theta}2}(\bar\oo\times[t,t+1])}\le C\ \ \ \ \ for\ all\ t\in(0,\yy).
\eess
\end{lem}

\begin{proof}[\bf Proof of Theorem \ref{t1.1}]
Thanks to \eqref{3.1g} and \eqref{2.4a} with $\beta>2$, we may apply Lemma \ref{l2.5} to show the global solvability of \eqref{2.1b}, and hence solve \eqref{1.1}.
\end{proof}

\section{Global generalized solution: proof of Theorem \ref{t1.2}}
\setcounter{equation}{0} {\setlength\arraycolsep{2pt}
In what follows, we always set $\beta=2$. We now introduce the following concept of global generalized solutions (cf. \cite{Black-M3AS2020,mw-2015siam,Winkler-M3AS2019}).
\begin{defi}\label{d1.1}
By a global generalized solution of \eqref{1.1} we mean a pair $(u,v,w)$ of nonnegative functions defined a.e. in $\oo\times(0,\yy)$ which are such that
\bess
 \left\{\begin{array}{lll}
 u\in L_{loc}^2([0,\yy);W^{1,2}(\oo)),\\[1mm]
 v\in L_{loc}^1(\bar\oo\times[0,\yy)),\\[1mm]
 w\in L_{loc}^\yy(\bar\oo\times[0,\yy))\cap L_{loc}^2([0,\yy);W^{1,2}(\oo))
 \end{array}\right.
 \eess
and
\bess
\nn\ln(v+1)\in L_{loc}^2(\bar\oo\times[0,\yy);\R^2),
\eess
and such that
\bes
-\int_0^\yy\ii u\vp_t-\ii u_0\vp(\cdot,0)&=&-\int_0^\yy\ii\nn u\cdot\nn\vp+\int_0^\yy\ii u\nn w\cdot\nn\vp,\label{1.4}
\ees
and
\bes
-\int_0^\yy\ii w\vp_t-\ii w_0\vp(\cdot,0)&=&-\int_0^\yy\ii\nn w\cdot\nn\vp+\int_0^\yy\ii (-(u+v)w-w+r)\vp\label{1.6}
\ees
hold for all $\vp\in C_0^\yy(\bar\oo\times[0,\yy))$, and the inequality
\bes
&&-\int_0^\yy\ii \ln(v+1)\psi_t-\ii \ln(v_0+1)\psi(\cdot,0)\nm\\
&\ge&\int_0^\yy\ii|\nn\ln(v+1)|^2\psi-\int_0^\yy\ii\nn \ln(v+1)\cdot\nn\psi-\int_0^\yy\ii \frac{v}{v+1}\kk(\nn u\cdot\nn\ln(v+1)\rr)\psi\nm\\
&&+\int_0^\yy\ii\frac{v}{v+1}\nn u\cdot\nn\psi+\int_0^\yy\ii \frac{v-v^2}{v+1}\psi,\label{1.5}
\ees
holds for every nonnegative $\psi\in C_0^\yy(\bar\oo\times[0,\yy))$ and
\bess
\ii u(\cdot,t)=\ii u_0\ \ and\ \ \ii v(\cdot,t)\le \ii v_0+\int_0^t\ii (v-v^2)\ \ \ \ \ for\ a.e.\ t>0.
\eess
\end{defi}

For $\beta=2$, since the regularized problem \eqref{2.1b} depends on $\ep$, we shall use $(u_\ep,v_\ep,w_\ep)$ and $T_{max,\ep}$ to denote the solution of \eqref{2.1b} and the maximal time of existence, respectively. We hence rewrite \eqref{2.1b} as
 \bes
 \left\{\begin{array}{lll}
 u_{\ep t}=\tr u_{\ep}-\nabla\cd(u_\ep F_\ep(u_{\ep})\nabla w_{\ep}),&x\in\Omega,\ \ t>0,\\
 v_{\ep t}=\tr v_{\ep}-\nabla\cd(v_{\ep}\nabla u_{\ep})+v_{\ep}(1-v_{\ep}),&x\in\Omega,\ \ t>0,\\
 w_{\ep t}=\tr w_{\ep}-F(u_\ep)w_\ep-F(v_\ep)w_\ep- w_{\ep}+r,&x\in\Omega,\ \ t>0,\\
 \frac{ \partial u_\ep}{\pl\nu}=\frac{ \partial v_\ep}{\pl\nu}=\frac{ \partial w_\ep}{\pl\nu}=0,\ \ &x\in\partial\Omega,\ t>0,\\
  u_\ep(x,0)=u_0(x),\ v_\ep(x,0)=v_0(x),\ w_\ep(x,0)=w_0(x),\ &x\in\Omega.
 \end{array}\right.\label{4.1}
 \ees
Lemmas \ref{l2.1}, \ref{l2.2}, \ref{l3.3}, \ref{l3.4} hold for \eqref{4.1} with $(u,v,w)$ and $T_{max}$ replaced by $(u_\ep,v_\ep,w_\ep)$ and $T_{max,\ep}$ and the estimations in Lemmas \ref{l2.2}, \ref{l3.3}, \ref{l3.4} are $\ep$-independent. Making use of the $L^\yy$ boundedness of $F_\ep$ and $w_\ep$, it can be shown that the solution of \eqref{4.1} is global.

\begin{lem}\label{l4.1}
Let $\beta=2$ and $r$ be a nonnegative function satisfying \eqref{1.7a} and \eqref{1.7}. For any $\ep\in(0,1)$, the solution $(u_\ep,v_\ep,w_\ep)$ of \eqref{4.1} obtained in Lemma \ref{l2.1} is global, i.e., $T_{max,\ep}=\yy$. Moreover, there exists $C_*>0$ such that for all $\ep\in(0,1)$
\bes
&&\ii u_\ep= \ii u_0\ \ \ \ \ \ \faa\ t\in(0,\yy),\label{4.2a}\\
&&\ii v_\ep\le C_*\ \ \ \ \ \ \faa\ t\in(0,\yy),\label{4.4a}\\
&&\|w_\ep(\cdot,t)\|_{L^\yy(\oo)}\le M\ \ \ \ \ \ \faa\ t\in(0,\yy),\label{4.5}\\
&&\ii|\nn w_\ep|^2\le C_*\ \ \ \ \ \ \faa\ t\in(0,\yy),\label{4.6}
\ees
and for any $T>0$, there exists $C(T)>0$ such that for all $\ep\in(0,1)$
\bes
&&\int_0^T\ii |\nn u_\ep|^2\le C(T),\label{4.3}\\
&&\int_0^T\ii v_\ep^2\le C(T),\label{4.4}\\
&&\int_0^T\ii|\tr w_\ep|^2\le C(T),\label{4.6a}
\ees
and for any $p>1$ one can find $K_p>0$ such that for any $\ep\in(0,1)$
\bes
&&\ii u_\ep^p\le K_p\ \ \ \ \ \ \faa\ t\in(0,\yy).\label{4.2}
\ees
\end{lem}
\begin{proof}
Since $0\le F_\ep(s)\le \frac{1}{\ep}$ for all $s\ge0$ and $\|w_\ep(\cdot,t)\|_{L^\yy(\oo)}\le M$ for all $t\in(0,T_{max,\ep})$, we can use the arguments in \cite[Lemma 3.4]{Black-M3AS2020} step by step to get the global solvability of \eqref{4.1}. The estimations \eqref{4.2a}-\eqref{4.2} follow from Lemmas \ref{l2.2}, \ref{l3.3}, \ref{l3.4} directly.
\end{proof}

On basis of \eqref{4.2a}-\eqref{4.2}, we obtain the following further regularity information on the solution.
\begin{lem}\label{l4.2}
Let $\beta=2$ and $r$ be a nonnegative function satisfying \eqref{1.7a} and \eqref{1.7}. For any $T>0$, one can find $C(T)>0$ such that for all $\ep\in(0,1)$ the solution of \eqref{4.1} satisfies
\bes
\int_0^T\ii\frac{|\nn v_\ep|^2}{(v_\ep+1)^2}&\le& C(T),\label{4.7}\\
\int_0^T\ii|\nn w_\ep|^4&\le& C(T),\label{4.7c}\\
\int_0^T\ii |u_\ep F_\ep'(u_\ep)\nn w_\ep|^3&\le&C(T),\label{4.7a}\\
\int_0^T\ii (F_\ep(u_\ep)+F_\ep(v_\ep))^2w_\ep^2&\le& C(T),\label{4.7d}\\
\int_0^T\ii \kk|\frac{v_\ep-v_\ep^2}{v_\ep+1}\rr|^2&\le& C(T).\label{4.7e}
\ees
\end{lem}
\begin{proof}
Similar to the proof of \cite[Lemma 4.8]{Black-M3AS2020}, we can use \eqref{4.3}, \eqref{4.4a} and \eqref{4.4} to get \eqref{4.7}.

By \eqref{4.5} and the Gagliardo-Nirenberg inequality, there exist $C_1,C_2>0$ such that
\bes
\|\nn w_\ep\|_{L^4(\oo)}&\le& C_1\|\tr w_\ep\|_{L^2(\oo)}^{1/2}\|w_\ep\|_{L^\yy(\oo)}^{1/2}+C_1\|w_\ep\|_{L^\yy(\oo)}\nm\\
&\le& C_2\|\tr w_\ep\|_{L^2(\oo)}^{1/2}+C_2.\nm
\ees
Making use of \eqref{4.6a}, we have from the above inequality that
\bess
\int_0^T\ii |\nn w_\ep|^4\le C_3(T)
\eess
for some $C_3(T)>0$. This gives \eqref{4.7c}. Recalling $0\le F_\ep'(u_\ep)\le 1$, by \eqref{4.2}, \eqref{4.7c} and Young's inequality, we obtain \eqref{4.7a}.

Making use of \eqref{4.5} and \eqref{4.2}, there is $C_4(T)>0$ such that
\bes
\int_0^T\ii (F_\ep(u_\ep)w_\ep)^2\le M^2\int_0^T\ii \kk(\frac{u_\ep}{1+\ep u_\ep}\rr)^2\le M^2\int_0^T\ii u_\ep^2\le C_4(T).\label{4.8}
\ees
Similarly, we have from \eqref{4.4} and \eqref{4.5} that, one can find $C_5(T)>0$ fulfilling
\bes
\int_0^T\ii (F_\ep(v_\ep)w_\ep)^2\le  C_5(T).\label{4.9}
\ees
In view of \eqref{4.8} and \eqref{4.9}, we infer \eqref{4.7d}.

The inequality \eqref{4.7e} can be easily deduced from \eqref{4.4} (cf. \cite[Lemma 5.3]{Black-M3AS2020}).
\end{proof}

With \eqref{4.6a}, \eqref{4.2} and \eqref{4.7c} at hand, we can use the idea of \cite[Lemma 8.2]{Black-M3AS2020} to improve the regularity property of $\nn u_\ep$.
\begin{lem}\label{l4.3a}
Assume that $\beta=2$ and $r$ is a nonnegative function satisfying \eqref{1.7a} and \eqref{1.7}. For $T>0$ there exists $C(T)>0$ such that for any $\ep\in(0,1)$ the solution component $u_\ep$ of \eqref{4.1} fulfills
\bes
\int_0^T\ii|\nn u_\ep|^3\le C(T)\label{4.16}
\ees
\end{lem}
\begin{proof}
According to \eqref{4.1}, the solution component $u_\ep$ satisfies
\bes
 \left\{\begin{array}{lll}
 u_{\ep t}=\tr u_{\ep}-(F_\ep'(u_\ep)+u_\ep F_\ep''(u_\ep))\nn u_\ep\cdot\nn w_\ep-u_\ep F_\ep'(u_\ep)\tr w_\ep,&x\in\Omega,\ \ t>0,\\
 \frac{ \partial u_\ep}{\pl\nu}=0,\ \ &x\in\partial\Omega,\ t>0,\\
  u_\ep(x,0)=u_0(x),\ &x\in\Omega.
 \end{array}\right.\label{4.16a}
 \ees
By the definition of $F_\ep$, it is easy to see that
\bess
F_\ep'(s)+sF_\ep''(s)=\frac{1}{(1+\ep s)^2}-\frac{2\ep s}{(1+\ep s)^3}=\frac{1-\ep s}{(1+\ep s)^3}\ \ \ \ \faa\ s\ge0,
\eess
and hence
\bess
\kk|F_\ep'(s)+sF_\ep''(s)\rr|=\kk|\frac{1-\ep s}{(1+\ep s)^3}\rr|\le\frac{1}{(1+\ep s)^2}\le 1\ \ \ \ \faa\ s\ge0.
\eess
This combined with the fact that $0\le F_\ep'(s)\le 1$ for all $s\ge0$ implies
\bes
\kk|-(F_\ep'(u_\ep)+u_\ep F_\ep''(u_\ep))\nn u_\ep\cdot\nn w_\ep-u_\ep F_\ep'(u_\ep)\tr w_\ep\rr|\le |\nn u_\ep\cdot\nn w_\ep|+|u_\ep \tr w_\ep|.\label{4.16b}
\ees
Applying the maximal Sobolev regularity theory (\cite{Giga-Sohr,Hieber-pruss}) to \eqref{4.16a} and using \eqref{4.16b} as well as \eqref{1.2}, one can find $C_1,C_2>0$ such that for any $T>0$ and $\ep\in(0,1)$
\bes
&&\int_0^T\|u_{\ep t}\|_{L^{\frac53}(\oo)}^{\frac53}+\int_0^T\|u_\ep\|_{W^{2,{\frac53}}(\oo)}^{\frac53}\nm\\
&\le& C_1\|u_0\|_{W^{2,{\frac53}}(\oo)}^{\frac53}+C_1\int_0^T\|\nn u_\ep\cdot\nn w_\ep\|_{L^{{\frac53}}(\oo)}^{\frac53}+C_1\int_0^T\|u_\ep \tr w_\ep\|_{L^{\frac53}(\oo)}^{\frac53}\nm\\
&\le& C_1\int_0^T\ii|\nn u_\ep\cdot\nn w_\ep|^{\frac53}+C_1\int_0^T\ii|u_\ep \tr w_\ep|^{\frac53}+C_2.\label{4.16c}
\ees
For the first two terms in the right hand side of \eqref{4.16c}, we use Young's inequality to get
\bes
\int_0^T\ii|\nn u_\ep\cdot\nn w_\ep|^{\frac53}\le \int_0^T\ii|\nn u_\ep|^{\frac{20}{7}}+\int_0^T\ii|\nn w_\ep|^4,\label{4.16d}
\ees
and
\bes
\int_0^T\ii|u_\ep \tr w_\ep|^{\frac53}\le \int_0^T\ii u_\ep^{10}+\int_0^T\ii|\tr w_\ep|^2.\label{4.16e}
\ees
Inserting \eqref{4.16d} and \eqref{4.16e} into \eqref{4.16c} implies
\bes
&&\int_0^T\|u_\ep\|_{W^{2,{\frac53}}(\oo)}^{\frac53}\nm\\
&\le& C_1\int_0^T\ii|\nn u_\ep|^{\frac{20}{7}}+C_1\int_0^T\ii|\nn w_\ep|^4+C_1\int_0^T\ii u_\ep^{10}+C_1\int_0^T\ii|\tr w_\ep|^2+C_2.\label{4.16f}
\ees
By the Gagliardo-Nirenberg inequality (\cite{Friedman}), there exists $C_3>0$ such that
\bess
\|\nn u_\ep\|_{L^{\frac{20}{7}}(\oo)}^{\frac{20}{7}}&\le& C_3\|u_\ep\|_{W^{2,{\frac53}}(\oo)}^{\frac{10}{7}}\|u_\ep\|_{L^{10}(\oo)}^{\frac{10}{7}}.
\eess
Plugging this into \eqref{4.16f} provides
\bess
\int_0^T\|u_\ep\|_{W^{2,{\frac53}}(\oo)}^{\frac53}
&\le& C_1C_3\int_0^T\|u_\ep\|_{W^{2,{\frac53}}(\oo)}^{\frac{10}{7}}\|u_\ep\|_{L^{10}(\oo)}^{\frac{10}{7}}\nm\\
&&+C_1\int_0^T\ii|\nn w_\ep|^4+C_1\int_0^T\ii u_\ep^{10}+C_1\int_0^T\ii|\tr w_\ep|^2+C_2,
\eess
in which we employ Young's inequality to get $C_4>0$ such that
 \bes
\int_0^T\|u_\ep\|_{W^{2,{\frac53}}(\oo)}^{\frac53}
&\le& \frac12\int_0^T\|u_\ep\|_{W^{2,{\frac53}}(\oo)}^{\frac53}+C_4\int_0^T\|u_\ep\|_{L^{10}(\oo)}^{10}\nm\\
&&+C_1\int_0^T\ii|\nn w_\ep|^4+C_1\int_0^T\ii u_\ep^{10}+C_1\int_0^T\ii|\tr w_\ep|^2+C_2\nm\\
&=&  \frac12\int_0^T\|u_\ep\|_{W^{2,{\frac53}}(\oo)}^{\frac53}+(C_4+C_1)\int_0^T\ii u_\ep^{10}\nm\\
&&+C_1\int_0^T\ii|\nn w_\ep|^4+C_1\int_0^T\ii|\tr w_\ep|^2+C_2.\label{4.16g}
\ees
The inequalities \eqref{4.6a}, \eqref{4.2} and \eqref{4.7c} in conjunction with \eqref{4.16g} yields that, for any $T>0$ there is $C_5(T)>0$ such that for all $\ep\in(0,1)$ the solution component $u_\ep$ satisfies
\bes
\int_0^T\|u_\ep\|_{W^{2,{\frac53}}(\oo)}^{\frac53}\le C_5(T).\label{4.16h}
\ees
Making use of the Gagliardo-Nirenberg inequality and \eqref{4.2}, there exist $C_6,C_7>0$ fulfilling
\bess
\|\nn u_\ep\|_{L^{3}(\oo)}^{3}&\le& C_6\|u_\ep\|_{W^{2,{\frac53}}(\oo)}^{\frac{3}{2}}\|u_\ep\|_{L^{15}(\oo)}^{\frac{3}{2}}\nm\\
&\le& \|u_\ep\|_{W^{2,{\frac53}}(\oo)}^{\frac{5}{3}}+C_7,
\eess
which combined with \eqref{4.16h} implies \eqref{4.16}.
\end{proof}

We next present some information on time regularity of the time derivatives in \eqref{4.1} which will be used in the subsequent compactness argument.
\begin{lem}\label{l4.3}
Let $\beta=2$ and suppose that $r$ is a nonnegative function fulfilling \eqref{1.7a} and \eqref{1.7}. For $T>0$ there exists $C(T)>0$ such that for any $\ep\in(0,1)$ the solution of \eqref{4.1} satisfies
\bes
\int_0^T\|u_{\ep t}\|_{(W^{2,2}(\oo))^*}\le C(T)\label{4.17}
\ees
and
\bes
\int_0^T\|\pl_t\ln(v_\ep+1)\|_{(W^{2,2}(\oo))^*}\le C(T)\label{4.17a}
\ees
as well as
\bes
\int_0^T\|w_{\ep t}\|_{(W^{2,2}(\oo))^*}\le C(T).\label{4.17b}
\ees
\end{lem}
\begin{proof}
Since the proof is quite straightforward and standard, we only give the sketch. From the discussion of \cite[Lemma 5.5]{Black-M3AS2020} (or \cite[Lemma 4.7]{Winkler-M3AS2019}), there exists $C>0$ such that
\bes
\int_0^T\|u_{\ep t}\|_{(W^{2,2}(\oo))^*}&\le& C\int_0^T\kk(\|\nn u_\ep\|_{L^2(\oo)}^2+\|u_\ep F_\ep'(u_\ep)\nn w_\ep\|_{L^2(\oo)}^2+1\rr),\label{4.18}\\
\int_0^T\|\pl_t\ln(v_\ep+1)\|_{(W^{2,2}(\oo))^*}&\le& C\int_0^T\kk(\ii\frac{|\nn v_\ep|^2}{(v_\ep+1)^2}+\ii|\nn u_\ep|^2+\ii\kk|\frac{v_\ep-v_\ep^2}{v_\ep+1}\rr| +1\rr),\label{4.19}\\
\int_0^T\|w_{\ep t}\|_{(W^{2,2}(\oo))^*}&\le& C\int_0^T\kk(\ii|\nn w_\ep|^2+\ii u_\ep+\ii v_\ep+1\rr).\label{4.20}
\ees
In view of \eqref{4.18} with \eqref{4.3} and \eqref{4.7a}, we get \eqref{4.17}. The inequality \eqref{4.17a} can be inferred by inserting \eqref{4.7}, \eqref{4.3} and \eqref{4.7e} into \eqref{4.19}. Plugging \eqref{4.2a}, \eqref{4.4a} and \eqref{4.6} into \eqref{4.20} implies \eqref{4.17b}.
\end{proof}

Thanks to the Aubin-Lions lemma, the $\ep$-independent estimates collected above enable us to construct a limit triple $(u,v,w)$ through a standard extraction procedure.
\begin{lem}\label{l5.4}
Let $\beta=2$, and assume that $r$ is a nonnegative function fulfilling \eqref{1.7a} and \eqref{1.7}. Then there exist $(\ep_j)_{j\in \mathbb{N}}\subset(0,1)$ with $\ep_j\searrow0$ as $j\rightarrow\yy$ and nonnegative functions $u,v$ and $w$ defined a. e. in $\oo\times(0,\yy)$ such that
\bes
u&\in& L^\yy((0,\yy);L^p(\oo))\ \ for\ any\ p>1\  \ and\ \ \nn u\in L_{loc}^2(\bar\oo\times[0,\yy);\R^2),\nm\\
v&\in& L^\yy((0,\yy);L^1(\oo))\cap L_{loc}^2(\bar\oo\times[0,\yy))\ \ and\ \ \nn \ln(v+1)\in L_{loc}^2(\bar\oo\times[0,\yy);\R^2),\nm\\
w&\in& L^\yy(\oo\times(0,\yy))\cap L_{loc}^2([0,\yy);W^{2,2}(\oo))\ \ and\ \ \nn w\in L_{loc}^4(\bar\oo\times[0,\yy);\R^2),\nm
\ees
and such that the solutions of \eqref{4.1} fulfill
\bes
&&u_\ep\rightarrow u\ \ \ in\ L_{loc}^2(\bar\oo\times[0,\yy))\ and\ a.\ e.\ in\ \oo\times(0,\yy),\label{4.23}\\
&&\nn u_\ep\rightarrow \nn u\ \ \ in\ L_{loc}^2(\bar\oo\times[0,\yy);\R^2),\label{4.25}\\
&&v_\ep\rightarrow v\ \ \ in\ L_{loc}^p(\bar\oo\times[0,\yy))\ for\ p\in[1,2)\ and\ a.\ e.\ in\ \oo\times(0,\yy),\label{4.27}\\
&&\ln(v_\ep+1)\rightharpoonup \ln(v+1)\ \ \ in\ L_{loc}^2([0,\yy);W^{1,2}(\oo)),\label{4.28}\\
&&v_\ep\rightharpoonup v\ \ \ in\ L_{loc}^2(\bar\oo\times[0,\yy)),\label{4.29}\\
&&\frac{v_\ep-v_\ep^2}{v_\ep+1}\rightarrow \frac{v-v^2}{v+1}\ \ \ in\ L_{loc}^1(\bar\oo\times[0,\yy)),\label{4.30}\\
&&w_\ep\rightarrow w\ \ \ in\ L_{loc}^2(\bar\oo\times[0,\yy))\ and\ a.\ e.\ in\ \oo\times(0,\yy),\label{4.31}\\
&&w_\ep\mathop{\rightharpoonup}\limits^{\star} w\ \ \ in\ L^\yy(\oo\times(0,\yy)),\label{4.32a}\\
&&\nn w_\ep\rightarrow \nn w\ \ \ in\ L_{loc}^2(\bar\oo\times[0,\yy);\R^2),\label{4.33}\\
&&(F_\ep(u_\ep)+F_\ep(v_\ep))w_\ep\rightarrow (u+v)w\ \ \ in\ L_{loc}^1(\bar\oo\times[0,\yy)),\label{4.34}\\
&&u_\ep F_\ep'(u_\ep)\nn w_\ep\rightarrow u\nn w\ \ \ in\ L_{loc}^2(\bar\oo\times[0,\yy);\R^2),\label{4.34a}
\ees
as $\ep=\ep_j\searrow0$. Moreover,
\bes
\ii u(\cdot,t)=\ii u_0\ \ for\ a.\ e.\ t>0,\label{4.35}
\ees
and
\bes
\ii v(\cdot,t)\le \ii v_0+\int_0^t\ii (v-v^2)\ \ for\ a.\ e.\ t>0.\label{4.36}
\ees
\end{lem}
\begin{proof}
Although the rigorous process can be found in \cite[Proposition 6.1]{Black-M3AS2020} and \cite[Lemma 4.8]{Winkler-M3AS2019}, we still give the details of the proof for the completeness. It follows from \eqref{4.2}, \eqref{4.3} and \eqref{4.17} that, for all $T>0$,
\bes
(u_\ep)_{\ep\in(0,1)}\ {\rm is\ bounded\ in}\ L^2([0,T);W^{1,2}(\oo)),\nm\\
(u_{\ep t})_{\ep\in(0,1)}\ {\rm is\ bounded\ in}\ L^1([0,T);(W^{2,2}(\oo))^*),\nm
\ees
which enable us to apply the Aubin-Lions lemma to find $(\ep_j)_{j\in\mathbb{N}}\subset(0,1)$ with $\ep_j\searrow0$ as $j\rightarrow\yy$ and a nonnegative function $u\in L^2_{loc}([0,\yy);W^{1,2}(\oo))$ such that \eqref{4.23} holds. Thanks to \eqref{4.16} and \eqref{4.23}, we may apply Vitali's theorem to get the strong convergence in \eqref{4.25}. By \eqref{4.2} and Fatou's lemma, it is easy to see that $u\in L^\yy((0,\yy);L^p(\oo))$ for any $p>1$. The identity in \eqref{4.35} holds due to \eqref{4.23} and \eqref{4.2a}.

We have from \eqref{4.4}, \eqref{4.7} and \eqref{4.17a} that
\bes
(\ln(v_\ep+1))_{\ep\in(0,1)}\ {\rm is\ bounded\ in}\ L^2([0,T);W^{1,2}(\oo)),\label{4.37}\\
(\pl_t\ln(v_\ep+1))_{\ep\in(0,1)}\ {\rm is\ bounded\ in}\ L^1([0,T);(W^{2,2}(\oo))^*).\nm
\ees
Again by the Aubin-Lions lemma, along a further subsequence we have $v_\ep\rightarrow v$ a.e. in $\oo\times(0,\yy)$, which combined with \eqref{4.4} and Vitali's theorem implies \eqref{4.27}. The weak convergence properties in \eqref{4.28} and \eqref{4.29} result from \eqref{4.37} and \eqref{4.4}, respectively. Making use of \eqref{4.7e}, \eqref{4.27} and Vitali's theorem, it arrives at \eqref{4.30}. For the derivation of \eqref{4.36}, we refer to \cite[Proposition 6.1]{Black-M3AS2020} and omit the details.

Similarly, according to \eqref{4.5}, \eqref{4.6} and \eqref{4.17b},
\bes
(w_\ep)_{\ep\in(0,1)}\ {\rm is\ bounded\ in}\ L^2([0,T);W^{1,2}(\oo)),\nm\\
(w_{\ep t})_{\ep\in(0,1)}\ {\rm is\ bounded\ in}\ L^1([0,T);(W^{2,2}(\oo))^*)\nm
\ees
for all $T>0$, which by the Aubin-Lions lemma yields \eqref{4.31}. From \eqref{4.5}, we further have \eqref{4.32a}. By using \eqref{4.31}, \eqref{4.7c} and Vitali's theorem we get \eqref{4.33}. The strong convergence in \eqref{4.34} and \eqref{4.34a} can be derived by using \eqref{4.7d}, \eqref{4.7a} and Vitali's theorem.
\end{proof}

We are now in the position to show Theorem \ref{t1.2}.
\begin{proof}[\bf Proof of Theorem \ref{t1.2}]
To complete the proof, we shall show that the limit triple $(u,v,w)$ obtained in Lemma \ref{l5.4} satisfies the requirements of Definition \ref{d1.1}. Since the regularity information of $(u,v,w)$ is included in Lemma \ref{l5.4}, it remains to prove the identities in \eqref{1.4} and \eqref{1.6} as well as the inequality in \eqref{1.5}. Testing the first and third equation in \eqref{4.1} against $\vp\in C_0^\yy(\bar\oo\times[0,\yy))$, we have
\bes
-\int_0^\yy\ii u_\ep\vp_t-\ii u_0\vp(\cdot,0)&=&-\int_0^\yy\ii\nn u_\ep\cdot\nn\vp+\int_0^\yy\ii u_\ep F_\ep'(u_\ep)\nn w_\ep\cdot\nn\vp\label{4.38}
\ees
and
\bes
-\int_0^\yy\ii w_\ep\vp_t-\ii w_0\vp(\cdot,0)&=&-\int_0^\yy\ii\nn w_\ep\cdot\nn\vp-\int_0^\yy\ii (F_\ep(u_\ep)+F_\ep(v_\ep))w_\ep\vp\nm\\
&&-\int_0^\yy\ii w_\ep\vp+\int_0^\yy\ii r\vp.\label{4.39}
\ees
On the basis of the convergence statements in \eqref{4.23}, \eqref{4.25}, \eqref{4.31}, \eqref{4.33}, \eqref{4.34} and \eqref{4.34a}, the identities in \eqref{1.4} and \eqref{1.6} result from \eqref{4.38} and \eqref{4.39} by taking $\ep=\ep_j\searrow0$ in each integral separately.

We proceed to test the second equation of \eqref{4.1} against an arbitrary nonnegative $\psi\in C_0^\yy(\bar\oo\times[0,\yy))$ to get
\bes
&&-\int_0^\yy\ii \ln(v_\ep+1)\psi_t-\ii \ln(v_0+1)\psi(\cdot,0)\nm\\
&=&\int_0^\yy\ii|\nn\ln(v_\ep+1)|^2\psi-\int_0^\yy\ii\nn \ln(v_\ep+1)\cdot\nn\psi-\int_0^\yy\ii \frac{v_\ep}{v_\ep+1}\kk(\nn u_\ep\cdot\nn\ln(v_\ep+1)\rr)\psi\nm\\
&&+\int_0^\yy\ii\frac{v_\ep}{v_\ep+1}\nn u_\ep\cdot\nn\psi+\int_0^\yy\ii \frac{v_\ep-v_\ep^2}{v_\ep+1}\psi.\label{4.40}
\ees
Taking $(\ep_j)_{j\in\mathbb{N}}$ from Lemma \ref{l5.4} and making use of \eqref{4.28} we have
\bes
-\int_0^\yy\ii \ln(v_\ep+1)\psi_t\rightarrow-\int_0^\yy\ii \ln(v+1)\psi_t\ \ \ \ \ {\rm as}\ \ep=\ep_j\searrow0\label{4.41}
\ees
and
\bes
-\int_0^\yy\ii\nn \ln(v_\ep+1)\cdot\nn\psi\rightarrow-\int_0^\yy\ii\nn \ln(v+1)\cdot\nn\psi\ \ \ \ \ {\rm as}\ \ep=\ep_j\searrow0.\label{4.42}
\ees
Thanks to \eqref{4.30}, there holds
\bes
\int_0^\yy\ii \frac{v_\ep-v_\ep^2}{v_\ep+1}\psi\rightarrow \int_0^\yy\ii \frac{v-v^2}{v+1}\psi\ \ \ \ \ {\rm as}\ \ep=\ep_j\searrow0.\label{4.43}
\ees
Making use of \eqref{4.27} and the dominated convergence theorem, it is easy to see that
\bess
\frac{v_\ep}{v_\ep+1}\rightarrow\frac{v}{v+1}\ \ {\rm in}\ L_{loc}^2(\bar\oo\times[0,\yy))\ \ \ \ \ \ {\rm as}\ \ep=\ep_j\searrow0,
\eess
which combined with \eqref{4.25} implies
\bes
\int_0^\yy\ii\frac{v_\ep}{v_\ep+1}\nn u_\ep\cdot\nn\psi
\rightarrow\int_0^\yy\ii\frac{v}{v+1}\nn u\cdot\nn\psi\ \ \ \ \ \ {\rm as}\ \ep=\ep_j\searrow0.\label{4.44}
\ees
Since $\kk|\frac{v_\ep}{v_\ep+1}\rr|\le1$ and $\frac{v_\ep}{v_\ep+1}\rightarrow\frac{v}{v+1}$ a.e. in $\oo\times(0,\yy)$ as $\ep=\ep_j\searrow0$, we involve \eqref{4.25} and \cite[Lemma A.4]{mw-2015siam} to derive that
\bess
\frac{v_\ep}{v_\ep+1}\nn u_\ep \rightarrow \frac{v}{v+1}\nn u\ \ {\rm in}\ L_{loc}^2(\bar\oo\times[0,\yy))\ \ \ \ \ \ {\rm as}\ \ep=\ep_j\searrow0.
\eess
This in conjunction with \eqref{4.28} says
\bes
\int_0^\yy\ii \frac{v_\ep}{v_\ep+1}\kk(\nn u_\ep\cdot\nn\ln(v_\ep+1)\rr)\psi
 \rightarrow \int_0^\yy\ii \frac{v}{v+1}\kk(\nn u\cdot\nn\ln(v+1)\rr)\psi\ \ \ {\rm as}\ \ep=\ep_j\searrow0.\label{4.45}
\ees
In light of the lower semicontinuity of the norm in $L^2(\oo\times(0,\yy);\R^2)$ with respect to weak convergence, from \eqref{4.28} we get
\bes
\liminf_{\ep=\ep_j\searrow0}\int_0^\yy\ii|\nn\ln(v_\ep+1)|^2\psi \ge \int_0^\yy\ii|\nn\ln(v+1)|^2\psi.\label{4.46}
\ees
The inequality in \eqref{1.5} follows from \eqref{4.40}-\eqref{4.46}.

Above all, the limit triple $(u,v,w)$ obtained in Lemma \ref{l5.4} is the generalized solution of \eqref{1.1}.
\end{proof}


\begin{thebibliography}{99}
\bibliographystyle{siam}
\setlength{\baselineskip}{15pt}

\bibitem{Black-M3AS2020}T. Black, {\it Global generalized solutions to a forager-exploiter model with superlinear degradation and their eventual regularity properties}, Math. Models Methods Appl. Sci., 30(6)(2020), 1075-1117.

\bibitem{cao-M3AS2020}X. Cao, {\it Global radial renormalized solution to a producer-scrounger model with singular sensitivities}, Math. Models Methods Appl. Sci., 30(6)(2020), 1119-1165.

\bibitem{cao-tao2021NARWA}X. Cao and Y. Tao, {\it Boundedness and stabilization enforced by mild saturation of taxis in a producer scrounger model}, Nonlinear Anal.: Real World Appl., 57(2021), 103189.

\bibitem{Friedman}A. Friedman, {\it Partial Differential Equations}, Holt, Rinehart \& Winston, New York, 1969.

\bibitem{Giga-Sohr}Y. Giga and H. Sohr, {\it Abstract $L^p$ estimates for the Cauchy problem with applications to the Navier-Stokes equations in exterior domains}, J. Funct. Anal., 102(1991), 72-94.

\bibitem{Hieber-pruss}M. Hieber and J. Pr\"{u}ss, {\it Heat kernels and maximal $L^p$-$L^q$ estimates for parabolic evolution equations}, Comm. Part. Differ. Eqs., 22(1997), 1647-1669.

\bibitem{JinW}H. Jin and Z. Wang, {\it Global stability of prey-taxis systems}, J.  Differential Equations, 262(2017), 1257-1290.

\bibitem{Liw-zamp2021}J. Li and Y.  Wang, {\it Asymptotic behavior in a doubly tactic resource consumption model with proliferation}, Z. Angew. Math. Phys., 72(1)(2021), 21, 17pp.

\bibitem{Liu-nwrwa2019}Y. Liu, {\it Global existence and boundedness of classical solutions to a forager-exploiter model with volume-filling effects}, Nonlinear Anal.: Real World Appl., 50(2019), 519-531.

\bibitem{liuz-zamp2020}Y. Liu and Y. Zhuang, {\it Boundedness in a high-dimensional forager-exploiter model with nonlinear resource consumption by two species}, Z. Angew. Math. Phys., 71(2020), 151, 18pp.

\bibitem{Osaki-NA2002}K. Osaki, T. Tsujikawa, A. Yagi and M. Mimura, {\it Exponential attractor for a chemotaxis-growth system of equations}, Nonlinear Anal., 51(2002): 119-144.

\bibitem{SunL-jmaa2020}C. Sun and Y. Li, {\it Global bounded solution to a chemotaxis-convection model of capillary-sprout growth during tumor angiogenesis}, J. Math. Anal. Appl., 495(2021), 124665.

\bibitem{Tania2012}N. Tania, B. Vanderlei, J.P. Heath and L. Edelstein-Keshet, {\it Role of social interactions in dynamic patterns of resource
patches and forager aggregation}, Proc. Natl. Acad. Sci. USA, 109(2012), 11228-11233.

\bibitem{Taow-jde-2012}Y. Tao and M. Winkler, {\it Eventual smoothness and stabilization of large-data solutions in a three-dimensional chemotaxis system with
consumption of chemoattractant}, J. Differential Equations, 252(2012), 2520-2543.

\bibitem{TaoW-2019-jde}Y. Tao and M. Winkler, {\it Global smooth solvability of a parabolic-elliptic nutrient taxis system in domains of arbitrary dimension}, J. Differential Equations, 267(2019), 388-406.

\bibitem{TaoW-2019-forager}Y. Tao and M. Winkler, {\it Large time behavior in a forager-exploiter model with different taxis strategies for two groups in search of food}, Math. Models Methods Appl. Sci., 29(11)(2019), 2151-2182.

\bibitem{W-jde2021}J. Wang, {\it Global existence and boundedness of a forager-exploiter system with nonlinear diffusions}, J. Differential Equations,  276(2021), 460-492.

\bibitem{wjp-arxiv2021}J. Wang, {\it Global existence and stabilization in a forager-exploiter model with general logistic sources}, arXiv:2108.00590v1, 2021.

\bibitem{WW-M3AS2020}J. Wang and M. Wang, {\it Global bounded solution of the higher-dimensional forager-exploiter model with/without growth sources}, Math. Models Methods Appl. Sci., 30(7)(2020), 1297-1323.

\bibitem{mw-cpde2010logistic}M. Winkler, {\it Boundedness in the higher-dimensional parabolic-parabolic chemotaxis system with logistic source}, Comm. Part. Differ. Eqs., 35(2010), 1516-1537.

\bibitem{mw-2012cpde}M. Winkler, {\it Global large-data solutions in a chemotaxis-(Navier-)Stokes system modeling cellular swimming in fluid drops}, Comm. Part. Differ. Eqs., 37(2012), 319-351.

\bibitem{mw-2015siam}M. Winkler, {\it Large-data global generalized solutions in a chemotaxis system with tensor-valued sensitivities}, SIAM J. Math. Anal., 47(4)(2015), 3092-3115.

\bibitem{mw-2017jde}M. Winkler, {\it Asymptotic homogenization in a three-dimensional nutrient taxis system involving food-supported proliferation}, J. Differential Equations, 263(2017), 4826-4869.

   \bibitem{Winkler-M3AS2019}M. Winkler, {\it Global generalized solutions to a multi-dimensional doubly tactic resource consumption model accounting
for social interactions}, Math. Models Methods Appl. Sci., 29(3)(2019), 373-418.

\bibitem{xiang-jde2015}T. Xiang, {\it Boundedness and global existence in the higher-dimensional parabolic-parabolic chemotaxis system with/without growth source}, J. Differential Equations, 258(2015), 4275-4323.

\bibitem{mu-2020dcds}L. Xu, C. Mu and Q. Xin, {\it Global boundedness of solutions to the two-dimensional forager-exploiter model with logistic source}, Discrete Contin. Dyn. Syst., 41(7)(2021), 3031-3043.

\end{thebibliography}
 \end{document}